\documentclass[a4paper,12pt]{article}

\usepackage{titlesec}
\usepackage{amsthm}
\usepackage{comment}
\usepackage{amsmath}
\usepackage{amssymb}
\usepackage{amsfonts}
\usepackage{mathrsfs} 
\usepackage{graphicx}
\usepackage[dvipsnames]{xcolor}
\usepackage{bm}
\usepackage{lmodern}

\newtheorem{theorem}{Theorem}
\newtheorem{lemma}[theorem]{Lemma}

\newtheorem{remark}[theorem]{Remark}

\newtheorem{example}[theorem]{Example}
\newtheorem{definition}[theorem]{Definition}

\newcommand{\C}{\mathbb{C}}
\newcommand{\A}{\mathbf{A}}
\newcommand{\B}{\mathbf{B}}
\newcommand{\R}{\mathbb{R}}
\newcommand{\NN}{\mathbb{N}}
\newcommand{\X}{\mathbf{X}}
\newcommand{\x}{\mathbf{x}}
\newcommand{\y}{\mathbf{y}}
\newcommand{\ab}{\mathbf{a}}
\newcommand{\bb}{\mathbf{b}}
\newcommand{\cb}{\mathbf{c}}
\newcommand{\db}{\mathbf{d}}
\newcommand{\di}{\text{d}}
\newcommand{\ib}{\mathbf{i}}

\DeclareMathOperator{\st}{s.t.}
\DeclareMathOperator{\re}{Re}
\DeclareMathOperator{\opt}{opt}
\DeclareMathOperator{\im}{Im}

\DeclareMathOperator{\Tr}{Tr}

\newcommand{\cQ}{\mathcal{Q}}
\newcommand{\cR}{\mathcal{R}}
\newcommand{\cP}{\mathcal{P}}
\newcommand{\cM}{\mathcal{M}}

\usepackage{braket}
\usepackage{hyperref}
\usepackage{cleveref}
\newcommand{\revise}[1]{{\color{black}#1}}

\title{\bf Approximating the order 2 quantum Wasserstein distance using the moment-SOS hierarchy}
\date{\today}
\author{Saroj Prasad Chhatoi\thanks{{\tt{spchhatoi@laas.fr}} LAAS-CNRS} \and Victor Magron\thanks{{\tt{vmagron@laas.fr}} LAAS-CNRS \& Institute of Mathematics from Toulouse, France}}
\begin{document}
\maketitle

\begin{abstract}
Optimal transport theory has recently been extended to quantum settings, where the density matrices generalize the probability measures. In this paper, we study the computational aspects of the order $2$ quantum Wasserstein distance, formulating it as an infinite dimensional linear program in the space of positive Borel measures supported on products of two unit spheres. 
This formulation is recognized as an instance of the Generalized Moment Problem, which enables us to use the moment-sums of squares hierarchy to provide a sequence of lower bounds converging to the distance. 
We illustrate our approach with numerical experiments. 
\end{abstract}

\section{Introduction}

Optimal transport theory has emerged as a fundamental tool in mathematics, providing a geometric and variational framework for comparing probability measures. In particular, the family of Wasserstein distances is central to this theory and have found widespread applications across 
mesh and image processing \cite{solomon2015convolutional}, as well as the study of nonlinear partial differential equations \cite{santambrogio2015ot,villani2008optimal}, among others. 

On the algorithmic side, several mature numerical schemes can be used. 
Besides the classical Kantorovich framework based on linear programming \cite{kantorovich1942}, one can rely on the fluid-dynamical Benamou–Brenier formulation solved by augmented Lagrangian methods \cite{benamou2000computational}, the Sinkhorn–Cuturi entropic scheme for large-scale problems \cite{cuturi2013sinkhorn}, and a host of multiscale, stochastic, and convex-dual variants that accelerate these solvers or exploit structure in the data; see the comprehensive survey of Peyré and Cuturi \cite{peyre2019computational} for details and further references. 

Over the past decade there has been intensive efforts to generalize this theory to quantum settings where the density matrices replace the probability measures. 
A natural strategy for defining quantum analogues of optimal transport is the notion of quantum couplings, i.e., joint quantum states with fixed marginals. 
\revise{We refer to \cite{feliciangeli2023non} for a basic recent literature review. }
In analogy to the classical Monge--Kantorovich framework, this leads to defining transport costs via a cost operator $C$, where the total cost of transporting a density matrix $\rho$ to $\nu$ is given by
\[
\inf_{\tau \in Q(\rho, \nu)} \Tr[C \tau]
\]
defined over the set $Q(\rho,\nu)$ of possible joint quantum states, also called \textit{quantum couplings}. 
However, this direct generalisation does not satisfy the desired Wasserstein distance properties. In \cite{nogoresultZhou}, the authors proved that there is no single cost operator that produces a sensible metric such as half the trace norm between states. Within this coupling-based paradigm, several alternative constructions have been proposed.  
The authors from \cite{DepalmaTrevisan2021} construct couplings as density operators on a bigger Hilbert space with marginals  $\rho^T$ and $\nu$. This formulation admits a canonical purification and naturally links couplings to quantum channels.  The cost matrix is built from a set of self-adjoint operators (quadratures), yielding a Wasserstein-type distance $W^S_2(\rho, \sigma)$ through second-order moments. While $W^S_2$ is not a metric in the strict sense, it satisfies modified triangle inequalities and approximates classical Wasserstein distances in semiclassical regimes. The authors in \cite{FriedlandCole2022} use as their cost operator the projector onto the antisymmetric subspace of the tensor-product Hilbert space formed by two identical copies of the system. This construction gives rise to a quantity $W_2^{\mathrm{asym}}$ that captures the antisymmetric content of joint couplings.
To recover data processing and monotonicity, a stabilized version $W_2^{\mathrm{asym},s}$ was introduced.  
\revise{As in the classical case, one can employ convex regularization \cite{caputo2024quantum} or von Neumann entropy regularization \cite{gerolin2023non} for unbalanced noncommutative optimal transport problems. }
For an overview of the various proposals for quantum Wasserstein distances and their properties, we refer the reader to the recent survey \cite{beatty2025wassersteindistancesquantumstructures}. 

Similar to the coupling method, a quantum analogue of the Benamou--Brenier formulation of the order $2$ Wasserstein distance was introduced in~\cite{CarlenMaas2017}. The definition is motivated by the classical property that the gradient flow of the relative entropy with respect to the Wasserstein metric gives rise to the Fokker--Planck equation. 

The approach presented in \cite{beatty2024} avoids cost matrices altogether by defining transport plans as collections of weighted pure-state pairs, and using any suitable metric on the projective Hilbert space to define a Wasserstein-type distance $W_p^d$. 

In this article we focus on the numerical implementation of $p=2,d=2$ case where the cost is quadratic in the pure states, and the quantum  Wasserstein distance problem leads to an optimization problem over separable states-- joint density \revise{operators} with fixed marginals which can be expressed as a mixture of product states. 
This structure establishes the connection to the Doherty–Parrilo–Spedalieri (DPS) hierarchy \cite{doherty2002distinguishing}. 
To test separability of quantum states, the DPS hierarchy provides a sequence of convex relaxations that impose symmetric extensions and positivity constraints that becomes increasingly tight at higher levels. 
Each relaxation bound can be computed by solving a \textit{semidefinite program}  \cite{vandenberghe1996semidefinite}, i.e., by minimizing a linear function under linear matrix inequality constraints. 

In \cite{gribling2022bounding}, the authors provide a moment perspective for this hierarchy and show that state separability can be cast as an instance of the \textit{Generalized Moment Problem} (GMP), which boils down to optimizing over the cone of positive Borel measures supported on the product of two unit spheres. 
\revise{We refer the interested reader to \cite{lasserre2008semidefinite} for a comprehensive introduction on the GMP, as well as for the practical approach based on semidefinite relaxations. }
%
%
Deeply inspired by this connection, we show that the order 2 quantum Wasserstein distance problem admits a similar GMP formulation, where one minimizes the squared Frobenius norm of the difference of rank-one projectors under marginal constraints. 
This allows us to use the \textit{moment-sums of squares} (abbreviated as moment-SOS)  hierarchy \cite{lasserre2001global,lasserre2009moments}, initially designed to solve polynomial optimization problems. 
Each level of the hierarchy yields a lower bound on the true transport cost and the convergence is guaranteed. The resulting semidefinite relaxations are computationally tractable for moderate state dimensions and hierarchy levels. 

A closely related development is the entanglement–marginal problem studied in \cite{navascues2021entanglement}. There the task is to decide whether a family of locally compatible reduced states can be extended to a separable global state. 
The authors provide a hierarchy of semidefinite relaxations that mirrors the DPS hierarchy but is tailored to marginal data: the constraints are imposed only on the neighbourhoods that actually appear in the instance. 
This sparsity-exploiting hierarchy reduces the size of the matrix variables involved in the relaxation. In translation-invariant lattice models, the memory requirements grow only linearly with the lattice width. 
As for sparse polynomial optimization \cite{lasserre2006convergent,magron2023sparse}, the hierarchy converges under additional chordality assumptions on the underlying sparsity pattern graph. 

\paragraph{Outline and contributions.} 
The rest of this note is organized as follows. 
In Section \ref{sec:prelim}, we recall standard results on polynomial optimization, and formulate the order 2 quantum Wasserstein distance. 
Then we characterize in Section~\ref{sec:measure_formulation} this distance as an infinite-dimensional linear program over positive Borel measures supported on the Cartesian product of two unit spheres. 
This linear program is a particular instance of the GMP, for which one can directly leverage the moment-SOS hierarchy to obtain a sequence of certified lower bounds of the global minimum with convergence guarantees. 
We explicitly prove convergence of this hierarchy of relaxations in Section~\ref{sec:sos_hierarchy}, and show that each relaxation has no duality gap.  
Eventually, we illustrate the practical use of this hierarchy on simple numerical examples and outline further research perspectives in Section~\ref{sec:numerics}. 

\section{Preliminaries and problem statement}
\label{sec:prelim}
\subsection{Preliminaries on polynomial optimization}
\label{sec:notationpop}
Let \( \x = (x_1, x_2, \ldots, x_n) \) and \( \overline{\x} = (\overline{x}_1, \overline{x}_2, \ldots, \overline{x}_n) \) denote vectors of complex variables and their conjugates, respectively. Let \( [\x, \overline{\x}] \) denote the vector of monomials of the form \( \x^{\alpha} \overline{\x}^{\beta} = \prod_i x_i^{\alpha_i} \prod_j \overline{x}_j^{\beta_j} \), where \( \alpha, \beta \in \mathbb{N}^n \) are multi-indices. Let \( [\x, \overline{\x}]_t \) be the subvector consisting of monomials with total degree at most \( t \), i.e., \( |\alpha| + |\beta| \leq t \). We define the space of polynomials with complex coefficients and degree at most $t$  as,
\begin{align*}
\C[\x,\overline{\x}]_t := \text{Span}\{m \: \vert\: m\in [\x,\overline{\x}]_t\} =\{ \Sigma_{m\in [\x,\overline{\x}]_t} a_m m:a_m \in \C\}.
\end{align*}
Given a polynomial $p = \sum_{\alpha,\beta} p_{\alpha,\beta} \x^{\alpha} \overline{\x}^{\beta}$, it conjugate is $\overline{p} = \sum_{\alpha,\beta} \overline{p}_{\alpha,\beta} \overline{\x}^{\alpha} \x^{\beta}$. 
A polynomial $p\in\C[\x,\overline{\x}]$ is called \textit{Hermitian} if $p=\overline{p}$. 
\revise{A \textit{Hermitian square} is any complex polynomial of the form $\overline{p} \, p$, and any finite sum of such elements is called a \textit{sum of Hermitian squares}. }
Let us denote by \( \Sigma[\x, \overline{\x}] \) the cone of sums of Hermitian squares of complex polynomials, and by \( \Sigma[\x, \overline{\x}]_{2t} \) its subset of elements with degree at most $2t$, for all $t \in \NN$. 

Given a positive $r \in \NN$, let \( \{g_1, \ldots, g_r\} \subseteq \mathbb{C}[\x, \overline{\x}] \) be a collection of Hermitian polynomials. 
We consider the basic closed semialgebraic set
\begin{equation}
    \X := \left\{ \x \in \C^n \,\middle|\, g_j(\x,\overline{\x}) \ge 0 \text{ for all } j = 1, \ldots, r \right\}.
\end{equation}

A particular subset of Hermitian polynomials that are nonnegative on \( \X \) is the quadratic module $\cQ(\X)$ associated with \( \X \), defined as
\[
\cQ(\X) := \left\{ \sum_{j=0}^r \sigma_j g_j \,\middle|\, \sigma_j \in \Sigma[\x, \overline{\x}] \right\}.
\]
Given a vector $\x \in \C^n$, its Euclidean norm is denoted by $\|\x\| = \sqrt{\sum_{i=1}^n x_i \overline{x}_i} = \sqrt{\sum_{i=1}^n |x_i|^2 }$. 
The quadratic module \( \cQ(\X) \) is said to be \textit{Archimedean} if there exists a scalar \( R > 0 \) such that $R - \|\x\|^2 \in \mathcal{Q}(\X)$. 

A linear functional $L: \C[\x,\overline{\x}] \to \C$ is called \textit{Hermitian} if $L(\overline{p})=\overline{L(p)}$, for all $p \in \C[\x,\overline{\x}]$. 
A Hermitian linear functional $L$ is called \textit{positive} if $L$ is nonnegative on $\Sigma[\x, \overline{\x}]$. 
 
For any $v \in \C^n$, let us denote by $L_{v} : \C[\x,\overline{\x}] \to \C$ the evaluation functional at $v$, defined by $L_v(p)=p(v,\overline{v})$, for every $p \in \C[\x,\overline{\x}]$. 
It is easy to see that $L_v$ is Hermitian and positive. 
Moreover $L_v$ has a representing atomic measure that is the Dirac measure concentrated at $v$, denoted by $\delta_v$.

Let us define $g_0:=1$. 
For a given integer \( t \geq \max_{1 \leq j \leq r} \left\lceil \deg(g_j)/2 \right\rceil \), the truncated quadratic module of order $2t$ associated to $\X$ is defined as
\[
\cQ(\X)_{2t} := \left\{ \sum_{j=0}^r \sigma_j g_j \;\middle|\; \sigma_j \in \Sigma[\x, \overline{\x}],\ \deg(\sigma_j g_j) \leq 2t \right\}.
\]
 
The following lemma, stated, e.g., in \cite[Lemma 7]{gribling2022bounding}, provides a uniform bound on the evaluation of monomials under linear functionals that are nonnegative on a given truncated quadratic module. 
\begin{lemma}\label{lem:moment_bound}
Let $\X$ be a compact semialgebraic set and assume that $R - \|\x\|^2 \in \cQ(\X)_2$. 
For any $t \in \NN$, let assume that $L_t : \C[\x,\overline{\x}]_{2t} \to \C$ is a linear functional being nonnegative on $\cQ(\X)_{2t}$. 
Then we have
\[
|L_t(w)| \leq  L_t(1) \cdot R^{|w|/2} \text{ for all } w \in [\x,\overline{\x}]_{2t}.
\]
Moreover, if 
\[
\sup_{t \in \NN} L_t(1) < \infty,
\]
then $\{L_t\}_{t \in \NN}$ has a point-wise converging subsequence in the space of linear functionals on $\C[\x,\overline{\x}]$. 
\end{lemma}

As in \cite{gribling2022bounding}, we rely on the following complex analog results by Putinar \cite{putinar1993} and Tchakaloff \cite{tchakaloff1957cubature} to derive our moment perspective of the order 2 quantum Wasserstein distance problem, stated later on in Section \ref{sec:qwstatement}. 
For more details, we refer the interested reader to \cite[Appendix A.2]{gribling2022bounding}, where the authors indicate how to derive the complex analog statements from the real ones.

\begin{theorem}\cite[Theorem 8]{gribling2022bounding}
\label{thm:measure_representation}
Let $\X$ be a basic semialgebraic set, and assume that $\cQ(\X)$ is Archimedean. 
Let us assume that a Hermitian linear functional $L: \C[\x,\overline{\x}] \to \C$ is  nonnegative on $\cQ(\X)$. Then, the following holds. 
\begin{itemize}
\item[(i)] $L$ has a representing measure $\mu$ supported on $\X$, i.e., we have $L(p) = \int p d\mu$  for all  $p \in \mathbb{C}[\x,\overline{\x}]$.
\item[(ii)] For any $k \in \NN$, there exists a linear functional $\hat{L} : \C[\x,\overline{\x}] \to \C$, which coincides with $L$ on $\C[\x,\overline{\x}]_k$, and has a finite atomic representing measure $\hat{\mu}$ supported on $\X$, i.e., 
\[\hat{\mu} = \sum_{\ell=1}^K \lambda_{\ell} \delta_{v_{\ell}},\]
for some integer $K \geq 1$, scalars $\lambda_1,\dots,\lambda_K > 0$ and vectors $v_1,\dots,v_K \in \X$. 
\end{itemize}
\end{theorem}
\subsection{Order 2 quantum Wasserstein distance}
\label{sec:qwstatement}
Given $n \in \NN$ and a matrix $\rho = (\rho_{ij})_{i,j=1}^n \in \C^{n \times n}$, the trace of $\rho$ is denoted by $\Tr(\rho):=\sum_{i=1}^n \rho_{ii}$, and its transpose conjugate is given by $\rho^* := (\overline{\rho_{ji}})_{i,j=1}^n$. 
Such a matrix is called \textit{Hermitian} if $\rho^* = \rho$. 
If in addition $\rho$ has nonnegative eigenvalues, then it is called \textit{positive semidefinite}. 
A \textit{normalized quantum state} $\rho \in \C^{n \times n}$ is  a Hermitian positive semidefinite matrix with \( \Tr(\rho) = 1 \). We denote the set of such states by $\mathcal{D}(\C^n)$. 
A normalized quantum state is called \textit{pure} if it has rank 1, and \textit{mixed} otherwise. 
\begin{definition}[\cite{beatty2024}]
Let \( n \in \mathbb{N} \), and let \( \rho, \nu \in \mathcal{D}(\C^n) \). A quantum transport plan between \( \rho \) and \( \nu \) is any finite set of triples \( \{ (\omega_\ell, u_\ell, v_\ell)\}_\ell  \) such that
\[
\sum_\ell \omega_\ell u_\ell u_\ell^* = \rho, \quad \sum_\ell \omega_\ell v_\ell v_\ell^* = \nu,
\]
where \( \omega_\ell > 0 \), \( \sum_\ell \omega_\ell = 1 \), and \( u_\ell, v_\ell \in \mathbb{C}^n \) with \( \|u_\ell\| = \|v_\ell\| = 1 \). 
The set of all such quantum transport plans is denoted by \( Q(\rho, \nu) \).
\end{definition}
We now consider two examples to illustrate this latter concept of quantum transport plan.

\begin{example}[Pure to pure]\label{ex:pure_pure}
Let us consider the pure states $\rho = \begin{pmatrix}
        1 & 0\\
        0 & 0
        \end{pmatrix}$ and $\nu=\frac{1}{2} \begin{pmatrix}
        1 & 1\\
        1 & 1
        \end{pmatrix}$. 
The set of quantum transport plans $Q(\rho,\nu)$ is the singleton   \[
        Q(\rho,\nu) = \left\{ \left(1, \begin{pmatrix} 1 \\ 0 \end{pmatrix}, \frac{1}{\sqrt{2}} \begin{pmatrix} 1 \\ 1 \end{pmatrix} \right) \right\}.
        \]
    \end{example}
    \begin{example}[Mixed to mixed]\label{ex:mixed_mixed}
Now let us consider the two mixed quantum states 
    $\rho = \begin{pmatrix}
    \frac{3}{4} & 0\\ 
    0 & \frac{1}{4}
    \end{pmatrix}$ and $\nu=\begin{pmatrix}
    \frac{1}{2} & 0 \\ 
    0 & \frac{1}{2}
    \end{pmatrix}$.
    One feasible transport plan is 
    \[
    \Big\{\;\left(\:\frac{3}{4},\begin{pmatrix}1\\0\end{pmatrix},\begin{pmatrix}0\\1\end{pmatrix}\right)\:,\:\left(\:\frac{1}{4},\frac{1}{\sqrt{2}}\begin{pmatrix}1\\1\end{pmatrix},\frac{1}{\sqrt{2}}\begin{pmatrix}1\\-1\end{pmatrix}\right)\;\Big\}.
    \] 
    \end{example}
    
\if{
\begin{remark} 
In contrast to the pure state case, the set \( Q(\rho, \nu) \) is not known explicitly in general. 
\end{remark}
}\fi

\begin{definition}[\cite{beatty2024}]\label{wasserstein_problem}
Given a quantum transport plan \( e = \{ (\omega_\ell, u_\ell, v_\ell)\}_\ell  \in Q(\rho, \nu) \), the associated order $2$ quantum transport cost is defined as
\[
T_2(e) = \sum_\ell \omega_\ell \Tr\left[ (u_\ell u_\ell^*- v_\ell  v_\ell^*)\overline{(u_\ell u_\ell^*- v_\ell v_\ell^*)}\right].
\]
The order $2$ quantum Wasserstein distance between \( \rho \) and \( \nu \) is given by
\[
W_2(\rho, \nu) := \left( \inf_{e \in Q(\rho, \nu)} T_2(e) \right)^{1/2}.
\]
\end{definition}

\begin{remark}
\revise{When the source and target are pure states, there is only one feasible transport plan.} 
In the pure to pure case from Example~\ref{ex:pure_pure}, the transport plan is unique and yields an exact evaluation of the Wasserstein distance \( W_2(\rho,\nu) = 1 \). 
This is in contrast to the mixed to mixed case from Example~\ref{ex:mixed_mixed}, where the provided feasible plan is not necessarily optimal and gives only an upper bound \( 1 \geq W_2(\rho,\nu) \). 
\end{remark}
\section{Main results}
\subsection{An instance of the Generalized Moment Problem}
\label{sec:measure_formulation}

In a very similar way as in \cite[Section 5.2]{gribling2022bounding}, we first show that the order 2 quantum Wasserstein distance can be formulated as a GMP instance. 
Let \( \X = \{ (\x, \y) \in \mathbb{C}^{2n} : \|\x\| = \|\y\| = 1 \} \) be the complex \textit{bi-sphere}, i.e., the Cartesian product of two complex unit spheres, and let \( \mathcal{M}(\X) \) denote the set of finite positive Borel measures supported on \( \X \). 
Let us \revise{denote} \( f(\x,\overline{\x},\y,\overline{\y}) = \mathrm{Tr}\left[(\x\x^* - \y\y^*)\,(\overline{\x\x^* - \y\y^*})\right] \in \C[\x,\overline{\x}, \y, \overline{\y}] \) and consider the following GMP instance:
\begin{equation}
\label{eq:measC}
\begin{aligned}
P(\rho,\nu) = \inf_{\mu \in \cM(\X)} \quad  & \int_{\X} f \di \, \mu   \\	
\st 
\quad & \int_{\X} \x \x^* \di \, \mu = \rho, \\
\quad & \int_{\X} \y \y^* \di \, \mu = \nu.
\end{aligned}
\end{equation}
The dual of \eqref{eq:measC} is given by:
\begin{equation}
\label{eq:dualC}
\begin{aligned}
D(\rho,\nu) = \sup_{\Lambda,\Gamma \in \C^{n \times n}} \quad & \Tr (\rho \Lambda + \nu \Gamma) \\
\mathrm{s.t.} \quad & f - \x^* \Lambda \x - \y^* \Gamma \y \geq 0 \quad \text{on } \X, \\
& \Lambda^* = \Lambda, \quad \Gamma^* = \Gamma.
\end{aligned}
\end{equation}

\begin{theorem}
\label{thm:qw}
Let $\rho, \nu \in \mathcal{D}(\C^n)$ and $f$ as above. 
The value of $P(\rho,\nu)$ is attained, and  $W^2_2(\rho,\nu) = P(\rho,\nu) = D(\rho,\nu)$. 
\end{theorem}
\begin{proof}
We first show that $W^2_2(\rho,\nu)\ge P(\rho,\nu)$. 
One can associate to any quantum transport plan $\{(\omega_{\ell},u_{\ell},v_{\ell})\}_{\ell} \in Q(\rho,\nu)$ a measure $\mu\in \cM(\X)$ given by $\mu = \sum_{\ell} \omega_{\ell} \delta_{(u_{\ell},v_{\ell})}$. 
The measure $\mu$ satisfies the equality constraints from \eqref{eq:measC} since 
\begin{gather*}
\int_{\X} \x\x^* \di \, \mu  = \sum_{\ell} \omega_{\ell} u_{\ell} u_{\ell}^* = \rho, \quad
\int_{\X}  \y\y^* \di \, \mu = \sum_{\ell} \omega_{\ell} v_{\ell} v_{\ell}^*  = \nu.
\end{gather*}
Similarly, the cost function is:
\begin{gather*}
    \int_{\X} f \di \mu = \sum_{\ell} \lambda_{\ell} \Tr[(u_{\ell} u_{\ell}^* - v_{\ell} v_{\ell}^*)(\overline{u_{\ell} u_{\ell}^* - v_{\ell} v_{\ell}^*})].
\end{gather*}
Hence, any feasible quantum transport plan yields a feasible measure  for $P(\rho,\nu)$ and we obtain $W_2^2(\rho,\nu)\ge P(\rho,\nu)$. 

Before proving the other direction, we provide an equivalent formulation of \eqref{eq:measC}-\eqref{eq:dualC} involving real Borel measures and real polynomials. 
Then, we will show that the inf is attained in the real GMP instance. 

The real reformulation is obtained similarly to \cite[Appendix A]{gribling2022bounding}, where the authors provide detailed explanation to transform polynomials in $\C[\x,\overline{\x},\y,\overline{\y}]$ into polynomials in $\R[\ab,\bb,\cb,\db]$, after applying the change of variables $\x = \ab + \ib \, \bb$, $\y = \cb + \ib \, \db$. 
For any $p \in \C[\x,\overline{\x},\y,\overline{\y}]$, let us define $ p_{\re} (\ab,\bb,\cb,\db) := \re(p(\ab+\ib \, \bb, \ab-\ib \, \bb, \cb+\ib \, \db, \cb-\ib \, \db ))$. 
For a given Hermitian $p$ one has $p(\x,\overline{\x},\y,\overline{\y}) = p_{\re} (\ab,\bb,\cb,\db)$, so
$f(\x,\overline{\x},\y,\overline{\y}) = f_{\re} (\ab,\bb,\cb,\db)$. 
The real analog of $\X$ is the set $\X_{\re} = \{ (\ab,\bb,\cb,\db) \in \R^{4n} : 1 - \|\ab\|^2 - \|\bb\|^2 = 1 - \|\cb\|^2 - \|\db\|^2 = 0  \}$. \\
In addition, the maps $\re(\cdot)$ and $\im(\cdot)$ act entrywise for vectors and matrices with polynomial entries in  $\C[\x,\overline{\x},\y,\overline{\y}]$. 
In particular one has $\re(\x \x^*) = \ab \ab^T + \bb \bb^T$ and $\im(\x \x^*) = \bb \ab^T - \ab \bb^T$. 
Then one easily shows that
\begin{equation}
\label{eq:measR}
\begin{aligned}
P(\rho,\nu) = \inf_{\mu^{\R}\in \cM(\X_{\re})} \quad  & \int_{\X_{\re}} f_{\re} \di \mu^{\R}   \\	
\text{s.t.}
\quad & \int_{\X_{\re}} (\ab \ab^T + \bb \bb^T) \di \mu^{\R} = \re(\rho), \\
\quad & \int_{\X_{\re}} (\bb \ab^T - \ab \bb^T) \di \mu^{\R} = \im(\rho), \\
\quad & \int_{\X_{\re}} (\cb \cb^T + \db \db^T) \di \mu^{\R} = \re(\nu), \\
\quad & \int_{\X_{\re}} (\db \cb^T - \cb \db^T) \di \mu^{\R} = \im(\nu), \\
\end{aligned}
\end{equation}
and
\begin{equation}
\label{eq:dualR}
\begin{aligned}
D(\rho,\nu) = \sup_{\Lambda_i,\Gamma_i \in \R^{n \times n}} \quad  & \Tr \left[\re(\rho) \Lambda_1 + \re(\nu) \Gamma_1 + \im(\rho) \Lambda_2 + \im(\nu) \Gamma_2\right] \\	
\text{s.t.}
\quad & f_{\re} \geq \Tr [(\ab \ab^T + \bb \bb^T)\Lambda_1 + (\cb \cb^T + \db \db^T)\Gamma_1 
 + (\bb \ab^T - \ab \bb^T)\Lambda_2 \\
\quad & \quad \quad \quad + (\db \cb^T - \cb \db^T)\Gamma_2 ] \text{ on } \X_{\re}, \\
\quad & \Lambda_1^T=\Lambda_1, \quad \Lambda_2^T=-\Lambda_2, \quad \Gamma_1^T=\Gamma_1, \quad \Gamma_2^T=-\Gamma_2. 
\end{aligned}
\end{equation}

We show that the inf is attained in \eqref{eq:measR}. 
The proof is readily the same as the one from, e.g., \cite[Lemma 2.2]{gamertsfelder2025effective}, and we provide a brief sketch:

\begin{itemize}
\item The GMP instance \eqref{eq:measR} has a non-empty feasible set because $Q(\rho,\nu)$ is non-empty: for any spectral decomposition $\rho = \sum_{\ell} \omega_{\ell} u_{\ell} u_{\ell}^*$ and $\nu = \sum_{j} \omega'_{j} v_j v_j^*$, one can obtain a finite sequence $\{(\omega_{\ell} \omega'_{j}, u_{\ell}, v_{j})\}_{\ell,j} \in Q(\rho,\nu)$.
%
%
%
\item The Banach-Alaoglu theorem (see, e.g., \cite[Theorem~3.5.16]{ash2014measure}) implies that the feasible set of \eqref{eq:measR} is compact in the weak-$\star$ topology. 
This is due to the fact that each solution $\mu$ of \eqref{eq:measC} must be a probability measure. 
Indeed, combining for instance the first equality constraint together with the fact that $\mu$ is supported on $\X$ yields
\[
\revise{
\int_{\X} \di \mu = 
\int_{\X} \|\x\|^2 \di \mu = 
\int_{\X} \Tr(\x \x^*) \di \mu  = 
\Tr(\rho)=1.}\] 
\end{itemize} 
Hence, the inf of the real GMP instance \eqref{eq:measR} is attained at a finite atomic measure, as a consequence of Tchakaloff's theorem \cite{tchakaloff1957cubature}. 
Equivalently, the inf of the complex GMP instance  \eqref{eq:measC} is also attained at a probability measure $\mu = \sum_{\ell} \omega_{\ell} \delta_{(u_{\ell},v_{\ell})}$.
This proves that $\{(\omega_{\ell},u_{\ell},v_{\ell})\}_{\ell} \in Q(\rho,\nu)$, and thus $W_2^2(\rho,\nu)\le P(\rho,\nu)$. 
\revise{Consequently, any optimal solution to the GMP instance \eqref{eq:measC} yields an optimal transport plan for the order 2 quantum Wasserstein distance problem.} 


Eventually, we prove that there is no duality gap between \eqref{eq:measC} and \eqref{eq:dualC}, or equivalently between \eqref{eq:measR} and \eqref{eq:dualR}. 
Let $\cP(\X_{\re})$ be the set of probability measures on $\X_{\re}$. 
Let us define a linear operator  $A : \cP(\X_{\re}) \to (\R^{2n \times 2n})^4$ as follows 
\[
A \mu := 
\begin{pmatrix}
& \int_{\X_{\re}}(\ab \ab^T + \bb \bb^T) \di \mu^{\R}  \\
&  \int_{\X_{\re}} (\bb \ab^T - \ab \bb^T) \di \mu^{\R} \\
&  \int_{\X_{\re}} (\cb \cb^T + \db \db^T) \di \mu^{\R} \\
&  \int_{\X_{\re}} (\db \cb^T - \cb \db^T) \di \mu^{\R}.
 \end{pmatrix}
 \]
The linear operator \(A \) is weak-$\star$ continuous and  the map \( \mu^{\R} \mapsto \int_{\X_{\re}} f_{\re} \di  \mu^{\R} \) is weak-$\star$ continuous as $f_{\re}$ is continuous on $\X_{\re}$. 
Thus the set \[
   \Bigl\{\bigl(\,\textstyle\int_{\X_{\re}}f_{\re} \di  \mu^{\R},\;
                       A \mu^{\R} \bigr)
          :\mu^{\R} \in \cP (\X_{\re})\Bigr\}
    \] is weak-$\star$ closed.  
    So the hypotheses of  \cite[Theorem 7.2, Chapter 4]{Barvinok2002ACI} are satisfied and thus we have no duality gap between \eqref{eq:measC} and \eqref{eq:dualC}.
\end{proof}
\revise{
Eventually, we show that the dual value $D(\rho,\nu)$ is attained if both source and target are positive definite quantum states. 
\begin{theorem}
\label{th:dual}
Let $\rho, \nu \in \mathcal{D}(\C^n)$ and $f$ as in Theorem \ref{thm:qw}. 
If both $\rho$ and $\nu$ are positive definite then the dual value $D(\rho,\nu)$ is attained. 
\end{theorem}
\begin{proof}
The proof follows from a compactness argument. 
Given a Hermitian matrix $\A \in \C^{n \times n}$, let us denote by $\|\A\| = \sqrt{\Tr(\A^2)}$ the Frobenius norm of $\A$, by $\lambda_{\min}(\A)$ and $\lambda_{\max}(\A)$ the minimal and maximal eigenvalues of $\A$, respectively. 
Let us introduce the cone
\[
K = \left\lbrace \left(\int \x \x^* \, \di \mu, \int \y \y^* \, \di \mu\right)  : \mu\in \mathcal{M}(\X)  \right\rbrace,  
\]
whose \emph{dual cone} $K^*$ is given by
\[
K^* = \bigl\{ (\A,\B) \in \C^{n \times n} \times  \C^{n \times n}  \;:\;  \A=\A^*, \B=\B^*, \Tr(\rho \A + \nu \B) \ge 0 \;\; \forall (\rho,\nu)\in K \bigr\}.
\]
Equivalently, using the definition of $K$, one can write
\[
K^* = \{ (\A,\B) \in \C^{n \times n} \times  \C^{n \times n} : \x^* \A \x + \y^* \B \y \ge 0 \quad \forall (\x,\y) \in \X \}.
\]
Let $\varepsilon := \min\{\lambda_{\min}(\rho), \lambda_{\min}(\nu)\} > 0$. 
Since both $\rho - \varepsilon I$ and $\nu - \varepsilon I$ are positive semidefinite, there exists an atomic measure $\mu \in \mathcal{M}(\X)$ such that $\int \x \x^* \, \di \mu = \rho - \varepsilon I$ and $\int \y \y^* \, \di \mu = \nu - \varepsilon I$, i.e., $(\rho - \varepsilon I, \nu - \varepsilon I) \in K$. 
Let $(\A,\B) \in K^*$. 
Duality implies that $\Tr((\rho - \varepsilon I) \A + (\nu - \varepsilon I) \B) \ge 0$, i.e., $\Tr(\rho \A + \nu \B) \ge \varepsilon \Tr(\A + \B)$. 
From the equivalent definition of $K^*$, notice that $(\A + \B)$ must be in particular positive semidefinite, since $\x^* \A \x + \y^* \B \y \ge 0$, for all $\x$. Then one has 
\[
\|\A + \B\| \leq \sqrt{n} \lambda_{\max} (\A + \B) \leq  \sqrt{n} \Tr(\A + \B), 
\]
which yields 
$\Tr(\rho \A + \nu \B) \ge \varepsilon/ \sqrt{n} \|\A + \B\|$. 

Let us now consider a feasible pair $(\Lambda,\Gamma)$ for \eqref{eq:dualC} and assume that the associated cost function is greater than 0. 
Let $c := \varepsilon/\sqrt{n}$.
By feasibility, one has $(I-\Lambda,I-\Gamma) \in K^*$, since $\x^* (I - \Lambda) \x + \y^* (I - \Gamma) \y = 2 - \x^* \Lambda \x - \y^* \Gamma \y \ge 0 \quad \forall (\x,\y) \in \X$. 
Therefore 
\[
\Tr(\rho (I - \Lambda) + \nu (I - \Lambda)) \geq c \|2 I - \Lambda - \Gamma\|. 
\]
In addition,  
\[
\Tr(\rho (I - \Lambda) + \nu (I - \Gamma))= 2 -  \Tr(\rho \Lambda + \nu \Gamma) \leq 2. 
\]
Therefore, $\|2 I -\Lambda -\Gamma\| \leq 2/c$, thus the Frobenius norm of $\Lambda+\Gamma$ is upper bounded by $M := 2/c + 2 \sqrt{n}$, and by Cauchy-Schwartz $|\Tr(\Lambda+\Gamma)| \leq \sqrt{n} M =: M' > 0$. 

Note that the tuple $(\Lambda + \alpha I, \Gamma - \alpha I)$ is also feasible for any $\alpha \in \R$ with the same objective value. 
Indeed, since $\Tr(\rho) = \Tr(\nu) = 1$, one has 
\[\Tr ((\rho\Lambda + \alpha I) + (\nu \Gamma - \alpha I)) = \Tr (\rho \Lambda + \nu \Gamma) + \alpha \Tr (\rho - \nu) = \Tr (\rho \Lambda + \nu \Gamma).\] 
In addition one has for all $(\x,\y) \in \X$
\[\x^* (\Lambda + \alpha I) \x + \y^* (\Gamma - \alpha I) \y = \x^* \Lambda \x + \y^* \Gamma \y + \alpha (\|\x\|^2 - \|\y\|^2) = \x^* \Lambda \x + \y^* \Gamma \y.\] 
Therefore, let us also assume without loss of generality that $\Tr(\Lambda)=0$, which shall allow one to bound all eigenvalues of $\Lambda$ and $\Gamma$. 

One has $-M' \leq \Tr(\Gamma) = \Tr(\Lambda+\Gamma) \leq 0$. 
For the latter inequality we used that for all $\x$ one has $f(\x,\overline{\x},\x,\overline{\x})=0$, so $\x (\Lambda + \Gamma) \x^* \leq 0$. 
In addition one has 
\begin{align*}
f(\x,\overline{\x},\y,\overline{\y})= \Tr[(\x \x^* - \y \y^*) \overline{(\x \x^* - \y \y^*)}] & = \left|\sum_{i=1}^n x_i^2\right|^2 + \left|\sum_{i=1}^n y_i^2\right|^2 - 2 \left|\sum_{i=1}^n x_i y_i\right|^2 \\
& \leq \left(\sum_{i=1}^n |x_i|^2\right)^2 + \left(\sum_{i=1}^n |y_i|^2\right)^2,
\end{align*}
which implies that $\x \Lambda \x^* + \y \Gamma \y^* \leq 2$, for all $(\x,\y) \in \X$, so in particular $ \lambda_{\max}(\Lambda) + \lambda_{\max} (\Gamma) \leq 2$. 
Since  $\Tr(\Lambda) = 0$, one has $\lambda_{\min}(\Lambda) \geq -(n-1) \lambda_{\max}(\Lambda)$.
Combining this with $\lambda_{\max}(\Gamma) \leq -\lambda_{\min} (\Lambda)$ and $ \lambda_{\max}(\Lambda) + \lambda_{\max} (\Gamma) \leq 2 $ yields $\lambda_{\max}(\Gamma) \leq 2 (n-1)/n$. 
In addition, 
\[
\lambda_{\min}(\Gamma) \geq -(n-1) \lambda_{\max}(\Gamma) + \Tr(\Gamma) \geq - 2 (n-1)^2/n - M'. 
\]
Eventually, one has $\lambda_{\max}(\Lambda) \leq -\lambda_{\min}(\Gamma) \leq  2 (n-1)^2/n + M'$ and $\lambda_{\min}(\Lambda) \geq - (n-1) \lambda_{\max}(\Lambda) \geq -2 (n-1)^3/n - (n-1) M'$. 
Consequently, all eigenvalues of such feasible matrix tuples are bounded in absolute value, which implies that all their entries' absolute values are also bounded. 

The feasible set  \eqref{eq:dualC} is the intersection of infinitely many polynomial superlevel sets. 
Therefore it is closed, which concludes the desired dual attainment proof. 
\end{proof}
}
\if{
\begin{remark}
\label{rk:dualnotatt}
We emphasize that the dual program \eqref{eq:dualC} might not be attained since the feasible set of \eqref{eq:dualC} is not compact. 
For any feasible solution $(\Lambda, \Gamma)$ of \eqref{eq:dualC}, the tuple $(\Lambda + \alpha I, \Gamma - \alpha I)$ is also feasible for any arbitrary large $\alpha \in \R$ with the same objective value. 
Indeed, since $\Tr(\rho) = \Tr(\nu) = 1$, one has $\Tr ((\rho\Lambda + \alpha I) + (\nu \Gamma - \alpha I)) = \Tr (\rho \Lambda + \nu \Gamma) + \alpha \Tr (\rho - \nu) = \Tr (\rho \Lambda + \nu \Gamma)$. 
In addition, $\x^* (\Lambda + \alpha I) \x + \y^* (\Gamma - \alpha I) \y = \x^* \Lambda \x + \y^* \Gamma \y + \alpha (\|\x\|^2 - \|\y\|^2) = \x^* \Lambda \x + \y^* \Gamma \y $, for all $(x,y) \in \X$. 
The above observation implies that the feasible set of \eqref{eq:dualC} is not bounded. 
\end{remark}
}\fi

\subsection{A hierarchy of moment-SOS relaxations}
\label{sec:sos_hierarchy}
We now approximate the problem defined in \eqref{eq:measC} using finite-dimensional semidefinite programs. 
We consider a hierarchy of moment-SOS relaxations of the GMP instance \eqref{eq:measC}. 
\revise{Note that the minimal relaxation order should be $t=2$ since the objective function has degree $4$. }
The hierarchy of moment relaxations, indexed by $t \geq 2$, is
%
\begin{equation}
\label{eq:momC}
\begin{aligned}
W^2_2(\rho,\sigma)_t := \inf_{L} \quad  & L(f)   \\	
\text{s.t.}
\quad & L: \C[\x,\overline{\x},\y,\overline{\y}]_{2t} \to \C  \text{ Hermitian linear}, \\
\quad & L(\x \x^*) = \rho, \\
\quad & L(\y \y^*) = \nu, \\
\quad & L \geq 0 \text{ on } \cQ(\X)_{2t}.
\end{aligned}
\end{equation}
The dual hierarchy of SOS programs, indexed by $t \geq 2$, is
\begin{equation}
\label{eq:sosC}
\begin{aligned}
D(\rho,\nu)_t = \sup_{\Lambda,\Gamma \in \C^{n \times n}} \quad & \Tr (\rho \Lambda + \nu \Gamma) \\
\mathrm{s.t.} \quad & f - \x^* \Lambda \x - \y^* \Gamma \y \in \cQ(\X)_{2t}, \\
& \Lambda^* = \Lambda, \quad \Gamma^* = \Gamma.
\end{aligned}
\end{equation}
For $t=\infty$, $W_2^2(\rho, \nu)_{\infty}$ involves linear functionals acting on the full polynomial space $\C[\x,\overline{\x},\y,\overline{\y}]$. 
Note that for all integer $t \geq 2$, the moment program \eqref{eq:momC} is a relaxation of \eqref{eq:measC}, thus $W_2^2(\rho, \nu)_t \leq W_2^2(\rho, \nu)$. 
In addition, one has $W_2^2(\rho, \nu)_t \geq D(\rho,\nu)_t$ by weak duality. 
The next result states that strong duality holds and that the moment-SOS hierarchy converges to the order 2 quantum Wasserstein distance. 
\begin{theorem}
Let $ \rho, \nu \in \mathcal{D}(\mathbb{C}^n)$. 
For all integer $t \geq 2$, there is no duality gap between \eqref{eq:momC} and \eqref{eq:sosC}, i.e., $W_2^2(\rho, \nu)_t = D(\rho,\nu)_t$. 
In addition,
\[
\lim_{t \to \infty} W_2^2(\rho, \nu)_t = W_2^2(\rho,\nu)_{\infty} = W_2^2(\rho, \nu).
\]
\end{theorem}

\begin{proof}
Let us fix an integer $t \geq 2$, and let us denote by $\cR_t$ and $\cR_t^{\opt}$ the feasible and optimal solution sets of \eqref{eq:momC}, respectively. 
We prove that there is no duality gap by using the (complex analog) result of \cite{trnovska2005strong}. 
To do so, we show that $\cR_t^{\opt}$ is nonempty and bounded. 

As already mentioned in the proof of Theorem \ref{thm:qw}, the set $Q(\rho, \nu)$ is nonempty: it contains a quantum transport plan $\{(\omega_{\ell},u_{\ell},v_{\ell})\}_{\ell}$ such that $\mu = \sum_{\ell} \omega_{\ell} \delta_{(u_{\ell},v_{\ell})}$ is feasible for \eqref{eq:measC}. 
The corresponding functional $L = \sum_{\ell} \omega_{\ell} L_{(u_{\ell},v_{\ell})} $ belongs to $\cR_t$, thus $\cR_t$ and $\cR_t^{\opt}$ are both nonempty. 

To prove that $\cR_t^{\opt}$ is bounded, we rely on the first result of Lemma \ref{lem:moment_bound}. 
Let $L_t \in \cR_t$. 
Note that $L_t(1) = L_t(\Tr(\x \x^*)) =  \Tr(\rho) = 1$. 
The assumptions of Lemma \ref{lem:moment_bound} are satisfied since $2 - \revise{\|\x\|^2} - \|\y\|^2 \in \cQ(\X)_2$ and $L_t$ is nonnegative on $\cQ(\X)_{2t}$. 
Then we have $|L_t(w)| \leq  2^{|w|/2}$ for all $ w \in [\x,\overline{\x},\y,\overline{\y}]_{2t}$. 
Thus the sets $\cR_t$ and $\cR_t^{\opt}$ are both bounded. 

The first claim then follows from the second sufficient condition stated in \cite[Corollary 1]{trnovska2005strong}. 

We now prove the second claim. 
As $\sup_{t \in \NN} L_t(1) = 1 < \infty$, the second result from Lemma \ref{lem:moment_bound} implies that there exists a Hermitian linear functional $L : \C[\x,\overline{\x},\y,\overline{\y}] \to \C$ which is the limit of a subsequence of $\{L_t\}_t$. 
Then $L$ is feasible for $W_2^2(\rho,\nu)_{\infty}$, which in turn implies that  $W_2^2(\rho,\nu)_{\infty} \leq L(f) = \lim_{t \to \infty} L_t(f) = \lim_{t \to \infty} W_2^2(\rho,\nu)_{t}$. 
Next, we prove that  $W_2^2(\rho, \nu) \leq W_2^2(\rho,\nu)_{\infty}$. 
Assume that $L$ is an optimal solution to $W_2^2(\rho,\nu)_{\infty}$. 
Since $\cQ(\X)$ is Archimedean and $L$ is nonnegative on $\cQ(\X)$, one can apply Theorem \ref{thm:measure_representation} with $k=4$ to conclude that the restriction of $L$ to $\C[\x,\overline{\x},\y,\overline{\y}]_4$ is a convex combination of evaluation functionals at points in $\X$. 
Namely, there exist scalars  $\omega_{\ell} > 0$ with $\sum_{\ell} \omega_{\ell} = 1$ and $(u_{\ell},v_{\ell}) \in \X$ such that $L(p) = \sum_{\ell} \omega_{\ell} p(u_{\ell}, \overline{u_{\ell}}, v_{\ell},  \overline{v_{\ell}} )$, for all $p \in \C[\x,\overline{\x},\y,\overline{\y}]_4$. 
In particular, one has $L(\x \x^*)= \sum_{\ell} \omega_{\ell} u_{\ell} u_{\ell}^* = \rho$,  $L(\y \y^*)= \sum_{\ell} \omega_{\ell} v_{\ell} v_{\ell}^* = \nu$ and $L(f) = \sum_{\ell} \omega_{\ell} \Tr[(u_{\ell} u_{\ell}^* - v_{\ell} v_{\ell}^*)(\overline{u_{\ell} u_{\ell}^* - v_{\ell} v_{\ell}^*})]$. 

This implies $\{(\omega_{\ell},u_{\ell},v_{\ell})\}_{\ell} \in Q(\rho,\nu)$ and  $W_2^2(\rho, \nu) \leq W_2^2(\rho,\nu)_{\infty}$, yielding the desired result. 
\end{proof}
\if{
\begin{remark}
Similarly to Remark \ref{rk:dualnotatt}, the dual SOS program \eqref{eq:sosC} is not attained. 
Indeed, for any feasible solution $(\Lambda, \Gamma)$ of \eqref{eq:sosC}, the tuple $(\Lambda + \alpha I, \Gamma - \alpha I)$ is also feasible for any arbitrary large $\alpha \in \R$ since $\alpha (\|x\|^2 - \|y\|^2) \in \cQ(X)_{2t}$ for any $t \in \NN$. 
\end{remark}
}\fi
\section{Concluding numerical experiments and perspectives}
\label{sec:numerics}
We first apply the moment-SOS hierarchy provided in Section \ref{sec:sos_hierarchy} on the single qubit systems ($n=2$) from Example~\ref{ex:pure_pure} and Example \ref{ex:mixed_mixed}. 
The lower bounds of the order 2 quantum Wasserstein distance are obtained by computing the values of the SOS program \eqref{eq:sosC} at the minimal relaxation order $t=2$. All semidefinite relaxations are modeled thanks to the \textsc{TSSOS} library \cite{magron2021tssos} (see also \cite[Appendix B]{magron2023sparse}),  implemented in \textsc{Julia} \cite{bezanson2017julia}. 
All results can be reproduced with online scripts\footnote{\url{https://wangjie212.github.io/TSSOS/dev/sos/\#Order-2-quantum-Wasserstein-distances}}. 
%
\titlespacing*{\paragraph}{0pt}{0pt}{1em}
\paragraph{Example~\ref{ex:pure_pure} continued.} 
Here, we consider the pair of pure states
 $\rho = \begin{pmatrix}
        1 & 0\\
        0 & 0
        \end{pmatrix}$ and $\nu=\begin{pmatrix}
        1 & 1\\
        1 & 1
        \end{pmatrix}$. 
For this pair an analytic calculation gives
\(W_2(\rho,\nu)=1\).
The SOS program \eqref{eq:sosC} at order~$t=2$ produces $W_{2}^2(\rho,\nu)_2 \simeq 1$ (up to 8 significant digits), thus certifying numerically that the relaxation is exact at the minimal relaxation order.
\paragraph{Example~\ref{ex:mixed_mixed} continued.} 
Here we consider the pair of mixed states 
$\rho = \begin{pmatrix}
        \frac{3}{4} & 0\\
        0 & \frac{1}{4}
        \end{pmatrix}$ and $\nu=\begin{pmatrix}
        \frac{1}{2} & 0\\
        0 & \frac{1}{2}
        \end{pmatrix}$.
The SOS programs \eqref{eq:sosC} at order~$t=2,3,4$  produce \(W_{2}^2(\rho,\nu)_t = 0.13397\), so one most likely expects the distance to be equal to $\sqrt{0.13397} \simeq 0.36601$. 
\paragraph{Example maximizing the Wasserstein distance.} 
Let us notice that 
\begin{align*}
f(\x,\overline{\x},\y,\overline{\y})= \Tr[(\x \x^* - \y \y^*) \overline{(\x \x^* - \y \y^*)}] & = \left|\sum_{i=1}^n x_i^2\right|^2 + \left|\sum_{i=1}^n y_i^2\right|^2 - 2 \left|\sum_{i=1}^n x_i y_i\right|^2 \\
& \leq \left(\sum_{i=1}^n |x_i|^2\right)^2 + \left(\sum_{i=1}^n |y_i|^2\right)^2 = 2,
\end{align*}
for all $(\x,\y) \in \X$. 
Therefore the objective function of \eqref{eq:measC} is upper bounded by $2$, and so is $W_2(\rho,\nu)^2$. 
With the pair of states $\rho = \frac{1}{2} \begin{pmatrix}
        1 & -1\\
        -1 & 1
        \end{pmatrix}$ and $\nu= \frac{1}{2} \begin{pmatrix}
        1 & 1\\
        1 & 1
        \end{pmatrix}$
we obtain $W_2^2(\rho,\nu)_2 \simeq 2$ (up to 8 significant digits), showing that the upper bound is actually saturated. 
\paragraph{Perspectives.}
A first topic of further investigation would be to reduce the computational effort required to solve the pair of semidefinite programs  \eqref{eq:momC}-\eqref{eq:sosC} at a fixed relaxation order $t$. 
For this, one could rely on the alternative  formulations proposed in \cite[Section 5.2]{gribling2022bounding} that allow us to block-diagonalize the matrix variable in various ways; see in particular \cite[Eq.~(61) and Lemma 22]{gribling2022bounding}. 
\revise{
In addition to this computational effort reduction, another practical goal is the extraction of optimal quantum transport plans. 
When a pseudo-moment matrix obtained at a given relaxation order satisfies the so-called \textit{flat extension} condition, one can conclude that the corresponding hierarchy has finite convergence. 
Under this condition, the algorithm from \cite{henrion2005detecting} outputs the atomic decomposition of the optimal measure. 
While this condition did not hold in our numerical experiments, one could derive another adapted one to moment relaxations with block-diagonalized matrix variables. 
}

Another future research topic is to analyze the convergence rate of our hierarchy. 
Convergence rates in $O(1/t^2)$ have been obtained for the DPS hierarchy in \cite{navascues2009power} and \cite{fang2021sum} by means of a quantum de Finetti theorem and the polynomial kernel method, respectively. 
In a follow-up paper \cite{magron2025convergence}, the second author has used such polynomial kernels to obtain a similar convergence rate for the order 2 quantum Wasserstein distance. 
It would be interesting to either retrieve or improve this rate by using a quantum de Finetti theorem. 

\revise{This note focused solely on the order $p=2$ quantum Wasserstein distance. 
Other cases of special interest include the infinite-order quantum Wasserstein distance ($p=\infty$). 
An exciting research track would be to identify for which values of $p$ one could  characterize the corresponding quantum Wasserstein distances as linear or more generally convex moment problems. 
Additionally, the recent framework \cite{klep2025sums} would allow us to also consider quantum transport problems cast as nonlinear polynomial moment problems. 
}
\paragraph{Acknowledgments.} 
\revise{The authors are grateful to the two anonymous referees for their constructive feedback. }
We thank Emily Beatty, Jonas Britz and Monique Laurent for fruitful discussions. 
Part of this research was performed while the first author was visiting the Institute for Pure and Applied Mathematics (IPAM), which is supported by the National Science Foundation (Grant No. DMS-1925919). 
This work was also supported by the European Union's HORIZON--MSCA-2023-DN-JD programme under the Horizon Europe (HORIZON) Marie Sklodowska-Curie Actions, grant agreement 101120296 (TENORS). 
\paragraph{Data availability statement.} We do not analyse or generate any datasets, because our work proceeds within a theoretical and mathematical approach. 
%


\begin{thebibliography}{FECidZ22}

\bibitem[Ash14]{ash2014measure}
Robert~B Ash.
\newblock {\em Measure, integration, and functional analysis}.
\newblock Academic Press, 2014.

\bibitem[Bar02]{Barvinok2002ACI}
A.I. Barvinok.
\newblock A course in convexity.
\newblock In {\em Graduate Studies in Mathematics}, 2002.

\bibitem[BB00]{benamou2000computational}
Jean-David Benamou and Yann Brenier.
\newblock A computational fluid mechanics solution to the monge--kantorovich
  mass transfer problem.
\newblock {\em Numerische Mathematik}, 84(3):375--393, 2000.

\bibitem[Bea25]{beatty2025wassersteindistancesquantumstructures}
Emily Beatty.
\newblock Wasserstein distances on quantum structures: an overview.
\newblock {\em arXiv preprint arXiv:2506.09794}, 2025.

\bibitem[BEKS17]{bezanson2017julia}
Jeff Bezanson, Alan Edelman, Stefan Karpinski, and Viral~B Shah.
\newblock Julia: A fresh approach to numerical computing.
\newblock {\em SIAM review}, 59(1):65--98, 2017.

\bibitem[BF24]{beatty2024}
Emily Beatty and Daniel~Stilck França.
\newblock {Order $p$ quantum Wasserstein distances from couplings}.
\newblock {\em arXiv preprint arXiv:2402.16477}, 2024.

\bibitem[CGMP24]{caputo2024quantum}
Emanuele Caputo, Augusto Gerolin, Nataliia Monina, and Lorenzo Portinale.
\newblock Quantum optimal transport with convex regularization.
\newblock {\em arXiv preprint arXiv:2409.03698}, 2024.

\bibitem[CM17]{CarlenMaas2017}
Eric~A. Carlen and Jan Maas.
\newblock Gradient flow and entropy inequalities for quantum markov semigroups
  with detailed balance.
\newblock {\em Journal of Functional Analysis}, 273(5):1810--1869, 2017.

\bibitem[Cut13]{cuturi2013sinkhorn}
Marco Cuturi.
\newblock Sinkhorn distances: Lightspeed computation of optimal transport.
\newblock In {\em Advances in Neural Information Processing Systems (NeurIPS)},
  pages 2292--2300, 2013.

\bibitem[DPS02]{doherty2002distinguishing}
Andrew~C Doherty, Pablo~A Parrilo, and Federico~M Spedalieri.
\newblock Distinguishing separable and entangled states.
\newblock {\em Physical Review Letters}, 88(18):187904, 2002.

\bibitem[DPT21]{DepalmaTrevisan2021}
Giacomo De~Palma and Dario Trevisan.
\newblock Quantum optimal transport with quantum channels.
\newblock {\em Ann. Henri Poincare}, 22(10):3199--3234, October 2021.

\bibitem[FECidZ22]{FriedlandCole2022}
Shmuel Friedland, Micha\l{} Eckstein, Sam Cole, and Karol \ifmmode~\dot{Z}\else
  \.{Z}\fi{}yczkowski.
\newblock Quantum monge-kantorovich problem and transport distance between
  density matrices.
\newblock {\em Phys. Rev. Lett.}, 129:110402, Sep 2022.

\bibitem[FF21]{fang2021sum}
Kun Fang and Hamza Fawzi.
\newblock The sum-of-squares hierarchy on the sphere and applications in
  quantum information theory.
\newblock {\em Mathematical Programming}, 190(1):331--360, 2021.

\bibitem[FGP23]{feliciangeli2023non}
Dario Feliciangeli, Augusto Gerolin, and Lorenzo Portinale.
\newblock A non-commutative entropic optimal transport approach to quantum
  composite systems at positive temperature.
\newblock {\em Journal of Functional Analysis}, 285(4):109963, 2023.

\bibitem[GLS22]{gribling2022bounding}
Sander Gribling, Monique Laurent, and Andries Steenkamp.
\newblock Bounding the separable rank via polynomial optimization.
\newblock {\em Linear Algebra and its Applications}, 648:1--55, 2022.

\bibitem[GM23]{gerolin2023non}
Augusto Gerolin and Nataliia Monina.
\newblock Non-commutative optimal transport for semi-definite positive
  matrices.
\newblock {\em arXiv preprint arXiv:2309.04846}, 2023.

\bibitem[GM25]{gamertsfelder2025effective}
Lucas Gamertsfelder and Bernard Mourrain.
\newblock The effective generalized moment problem.
\newblock {\em arXiv preprint arXiv:2501.09385}, 2025.

\bibitem[HL05]{henrion2005detecting}
Didier Henrion and Jean-Bernard Lasserre.
\newblock Detecting global optimality and extracting solutions in gloptipoly.
\newblock In {\em Positive polynomials in control}, pages 293--310. Springer,
  2005.

\bibitem[Kan42]{kantorovich1942}
Leonid~V. Kantorovich.
\newblock On the transfer of masses.
\newblock {\em Doklady Akademii Nauk}, 37:199--201, 1942.
\newblock English translation: \emph{Management Science}, 1960.

\bibitem[KMV25]{klep2025sums}
Igor Klep, Victor Magron, and Jurij Vol{\v{c}}i{\v{c}}.
\newblock Sums of squares certificates for polynomial moment inequalities.
\newblock {\em Foundations of Computational Mathematics}, pages 1--43, 2025.

\bibitem[Las01]{lasserre2001global}
Jean~B Lasserre.
\newblock Global optimization with polynomials and the problem of moments.
\newblock {\em SIAM Journal on optimization}, 11(3):796--817, 2001.

\bibitem[Las06]{lasserre2006convergent}
Jean~B Lasserre.
\newblock {Convergent SDP-relaxations in polynomial optimization with
  sparsity}.
\newblock {\em SIAM Journal on optimization}, 17(3):822--843, 2006.

\bibitem[Las08]{lasserre2008semidefinite}
Jean~B Lasserre.
\newblock A semidefinite programming approach to the generalized problem of
  moments.
\newblock {\em Mathematical Programming}, 112(1):65--92, 2008.

\bibitem[Las09]{lasserre2009moments}
Jean~Bernard Lasserre.
\newblock {\em Moments, positive polynomials and their applications}, volume~1.
\newblock World Scientific, 2009.

\bibitem[Mag25]{magron2025convergence}
Victor Magron.
\newblock Convergence rates for polynomial optimization on set products.
\newblock {\em arXiv preprint arXiv:2505.18580}, 2025.

\bibitem[MW21]{magron2021tssos}
Victor Magron and Jie Wang.
\newblock {TSSOS: a Julia library to exploit sparsity for large-scale
  polynomial optimization}.
\newblock {\em arXiv preprint arXiv:2103.00915}, 2021.

\bibitem[MW23]{magron2023sparse}
Victor Magron and Jie Wang.
\newblock {\em Sparse polynomial optimization: theory and practice}.
\newblock World Scientific, 2023.

\bibitem[NBA21]{navascues2021entanglement}
Miguel Navascu{\'e}s, Flavio Baccari, and Antonio Acin.
\newblock Entanglement marginal problems.
\newblock {\em Quantum}, 5:589, 2021.

\bibitem[NOP09]{navascues2009power}
Miguel Navascu{\'e}s, Masaki Owari, and Martin~B Plenio.
\newblock Power of symmetric extensions for entanglement detection.
\newblock {\em Physical Review A—Atomic, Molecular, and Optical Physics},
  80(5):052306, 2009.

\bibitem[PC{\etalchar{+}}19]{peyre2019computational}
Gabriel Peyr{\'e}, Marco Cuturi, et~al.
\newblock {Computational optimal transport: With applications to data science}.
\newblock {\em Foundations and Trends{\textregistered} in Machine Learning},
  11(5-6):355--607, 2019.

\bibitem[Put93]{putinar1993}
Mihai Putinar.
\newblock Positive polynomials on compact semi-algebraic sets.
\newblock {\em Indiana University Mathematics Journal}, 42(3):969--984, 1993.

\bibitem[San15]{santambrogio2015ot}
Filippo Santambrogio.
\newblock {\em Optimal Transport for Applied Mathematicians}, volume~87 of {\em
  Progress in Nonlinear Differential Equations and Their Applications}.
\newblock Birkh{\"a}user, 2015.

\bibitem[SdGP{\etalchar{+}}15]{solomon2015convolutional}
Justin Solomon, Fernando de~Goes, Gabriel Peyr{\'e}, Marco Cuturi, Adrian
  Butscher, Andy Nguyen, Tao Du, and Leonidas Guibas.
\newblock Convolutional wasserstein distances: Efficient optimal transportation
  on geometric domains.
\newblock {\em ACM Transactions on Graphics}, 34(4):66:1--66:11, 2015.

\bibitem[Tch57]{tchakaloff1957cubature}
V.~Tchakaloff.
\newblock Formules de cubature mécanique à coefficients non négatifs.
\newblock {\em Bulletin des Sciences Mathématiques}, 81:123--134, 1957.

\bibitem[Trn05]{trnovska2005strong}
Maria Trnovska.
\newblock Strong duality conditions in semidefinite programming.
\newblock {\em Journal of Electrical Engineering}, 56(12):1--5, 2005.

\bibitem[VB96]{vandenberghe1996semidefinite}
Lieven Vandenberghe and Stephen Boyd.
\newblock Semidefinite programming.
\newblock {\em SIAM review}, 38(1):49--95, 1996.

\bibitem[Vil08]{villani2008optimal}
C{\'e}dric Villani.
\newblock {\em Optimal Transport: Old and New}, volume 338 of {\em Grundlehren
  der mathematischen Wissenschaften}.
\newblock Springer Science \& Business Media, 2008.

\bibitem[ZYYY22]{nogoresultZhou}
Li~Zhou, Nengkun Yu, Shenggang Ying, and Mingsheng Ying.
\newblock Quantum earth mover’s distance, a no-go quantum
  kantorovich–rubinstein theorem, and quantum marginal problem.
\newblock {\em Journal of Mathematical Physics}, 63(10):102201, 10 2022.

\end{thebibliography}
\newcommand{\etalchar}[1]{$^{#1}$}

\end{document}